\documentclass[12pt]{amsart}

\usepackage{epsf, amssymb, graphicx, multicol, amscd,times,stmaryrd, amsmath}

\subjclass[2000]{Primary 53D12; Secondary 53D42.}

\newtheorem{theorem}{Theorem}[section]
\newtheorem{lemma}[theorem]{Lemma}

\newtheorem{proposition}[theorem]{Proposition}

\theoremstyle{remark}
\newtheorem{remark}[theorem]{Remark}

\theoremstyle{example} \theoremstyle{definition}
\newtheorem{definition}[theorem]{Definition}

\newcommand{\C}{\mathbb{C}}
\newcommand{\R}{\mathbb{R}}
\newcommand{\Z}{\mathbb{Z}}

\renewcommand{\L}{\mathcal{L}}

\newcommand{\ve}{\varepsilon}

\numberwithin{equation}{section}

\title{A note on the front  spinning construction}

\author{Roman Golovko}

\address{Universit\'{e} Paris-Sud, D\'{e}partement de Math\'{e}matiques, Bat. 425, 91405
Orsay Cedex, France} \email{roman.golovko@math.u-psud.fr}

\date{\today}
\keywords{Legendrian submanifold, Lagrangian cobordism, Legendrian
contact homology}

\begin{document}

\maketitle

\begin{abstract}
In this paper we introduce a notion of front $S^m$-spinning for Legendrian submanifolds of $\R^{2n+1}$. It generalizes the notion of front $S^1$-spinning which was invented by Ekholm, Etnyre and Sullivan.
We use it to prove that there are infinitely many pairs of exact Lagrangian cobordant and not pairwise
Legendrian isotopic Legendrian $S^1\times S^{i_{1}}\times \dots \times S^{i_{k}}$ which  have the same classical invariants if one of $i_{j}$'s is odd.
\end{abstract}

\section{Introduction and Main Results}
\subsection*{Basic definitions}
\emph{Standard contact $(2n+1)$-dimensional space} is Euclidean space $\R^{2n+1}$ equipped with the completely non-integrable field of hyperplanes $\xi = \ker \alpha$, where $\alpha$ is the contact $1$-form $\alpha=dz-\sum_{i=1}^{n}y_{i}dx_{i}$ in Euclidean coordinates $(x_{1}, y_{1},\dots,x_{n},y_{n},z)$.
The \emph{Reeb vector field} $R_{\alpha}$ which corresponds to $\alpha=dz-\sum_{i=1}^{n}y_{i}dx_{i}$ is given by $R_{\alpha}=\frac{\partial}{\partial z}$. An immersion of an $n$-manifold into $\R^{2n+1}$ is \emph{Legendrian} if it is everywhere tangent to the hyperplane field $\xi$, and the image of a Legendrian embedding is a \emph{Legendrian submanifold}. The \emph{Reeb chords} of a Legendrian submanifold $\Lambda$ are segments of
flow lines of $R_{\alpha}$ starting and ending at points of $\Lambda$.    The \emph{symplectization} of $\R^{2n+1}$ is the exact symplectic manifold
$(\R\times\R^{2n+1}, d(e^{t}\alpha))$, where $t$ is a coordinate on $\R$.
There are two natural projections
\begin{align*}
&\Pi_{F}(x_{1}, y_{1},\dots , x_{n}, y_{n}, z)=(x_{1}, \dots, x_{n},z)\quad \mbox{and}\\
&\Pi_{L}(x_{1}, y_{1},\dots , x_{n}, y_{n}, z)=(x_{1}, y_{1},\dots, x_{n}, y_{n})
\end{align*}
that we call the \emph{front projection} and the \emph{Lagrangian projection}, respectively.
The Lagrangian projection
$\Pi_{L}(\Lambda)$ of a Legendrian submanifold $\Lambda$ is an exact Lagrangian immersion into $\R^{2n}$.
Note that in the generic situation, i.e. for $\Lambda$ in an open dense subset of all Legendrian submanifolds with $C^{\infty}$ topology,
the self-intersection of $\Pi_{L}(\Lambda)$ consists of a finite number of transverse double points. These points correspond to all Reeb
chords of $\Lambda$.
A Legendrian submanifold is called \emph{chord generic} if it has a finite number of Reeb chords.
From now on we assume that all Legendrian submanifolds of $\R^{2n+1}$ are connected and chord generic.

\subsection*{Classical and non-classical invariants of Legendrian submanifolds}
There are two classical invariants of a closed, orientable Legendrian submanifold $\Lambda\subset \R^{2n+1}$, namely the \emph{Thurston--Bennequin invariant} (number) and the \emph{rotation class}.

The Thurston--Bennequin invariant  was originally defined  by Bennequin, see~\cite{Bennequin}, and independently by Thurston for Legendrian knots in $\R^{3}$, and then was generalized to higher dimensions by Tabachnikov~\cite{Tabachnikov}. The Thurston--Bennequin number $tb(\Lambda)$ of a closed, oriented
Legendrian $\Lambda\subset \R^{2n+1}$  is the linking number $lk(\Lambda,\Lambda')$, where $\Lambda'$ is an oriented submanifold obtained from $\Lambda$ by a small shift in the direction of $R_{\alpha}$.

The rotation class $r(\Lambda)$ was defined by Ekholm, Etnyre and Sullivan for all $n\geq 1$ in \cite{EkholmEtnyreSullivan} and is equal to the homotopy class of $(f, df_{\C})$ in the space of complex fiberwise isomorphisms $T\Lambda\otimes \C \to \xi$, where $f:\Lambda\to \R^{2n+1}$ is an
embedding of $\Lambda$. Note that if $\Lambda=S^n$ and $n$ is odd, then $r(\Lambda)\in \pi_{n}(U(n))\simeq \Z$ and we call $r(\Lambda)$ the \emph{rotation number}.

\emph{Legendrian contact homology} is a non-classical invariant of a closed, orientable Legendrian submanifold of $\R^{2n+1}$.
It was constructed in \cite{Chekanov} for the case when $n=1$ and in \cite{EkholmEtnyreSullivan2} for all $n\geq 1$.
The Legendrian contact homology of a closed, orientable Legendrian submanifold $\Lambda$ with the finite set of Reeb chords $\mathcal C_{\Lambda}$ is the homology of the noncommutative differential graded algebra $(\mathcal A_{\Lambda}, \partial_{\Lambda})$ over $\Z_{2}$ freely generated by the elements of $\mathcal C_{\Lambda}$ and is denoted by  $LCH_{\ast}(\Lambda)$.  The differential $\partial_{\Lambda}$ counts holomorphic curves in the symplectisation of $\R^{2n+1}$ whose domains are disks with points removed on the boundary. At these points, the holomorphic curve has one positive asymptotic and several negative asymptotics. For more details we refer to \cite{EkholmEtnyreSullivan2}.

Note that $(\mathcal A_{\Lambda}, \partial_{\Lambda})$  and even its homology may be infinite dimensional and hence it is difficult to use it for practical applications.
One of the ways to extract useful information from $(\mathcal A_{\Lambda}, \partial_{\Lambda})$ is to follow Chekanov's method of linearization.
An augmentation $\ve$ is an algebra homomorphism from $(\mathcal A_{\Lambda}, \partial_{\Lambda})$ to  $(\Z_{2}, 0)$ which satisfies $\ve(1) = 1$ and $\ve\circ \partial_{\Lambda} = 0$ and allows us to linearize the differential graded algebra to a finite dimensional complex $LC^{\ve}:=(A_{\Lambda}, \partial^{\ve}_{1, \Lambda})$ with homology groups $LCH^{\ve}_{\ast}(\Lambda)$. Here $A_{\Lambda}$ is the vector space over $\Z_{2}$ generated by the elements of $\mathcal C_{\Lambda}$. We let $LCH^{\ast}_{\ve}(\Lambda)$ be the homology of the dual complex $LC_{\ve}(\Lambda):=Hom(LC^{\ve}(\Lambda),\Z_{2})$. The linearized homology (cohomology) groups may depend on the choice of $\ve$. However, the set of graded groups $\{ LCH^{\ve}_{\ast} (\Lambda) \}$ ($\{ LCH_{\ve}^{\ast} (\Lambda) \}$), where $\ve$ is any augmentation of $(\mathcal A_{\Lambda}, \partial_{\Lambda})$, provides a Legendrian isotopy invariant, see \cite{Chekanov}.
It is shown in Section 3 of~\cite{EkholmSFTF2} that an exact Lagrangian filling of $\Lambda$ induces an augmentation of its Legendrian contact homology differential graded algebra.

In \cite{EkholmEtnyreSullivan}, Ekholm, Etnyre and Sullivan used Legendrian contact homology  to prove that for any $n>1$ there is an infinite family of Legendrian embeddings of the $n$-sphere into $\R^{2n+1}$ that are not Legendrian isotopic even though they have the same classical invariants. They also prove  similar results for Legendrian surfaces and $n$-tori, see \cite{EkholmEtnyreSullivan}. These results indicate that the theory of Legendrian submanifolds of standard contact $(2n+1)$-space is very rich.

\subsection*{Main results}
Observe that the family of $n$-tori from \cite{EkholmEtnyreSullivan} is constructed using the front $S^1$-spinning, which is a procedure to produce a closed, orientable Legendrian submanifold $\Sigma_{S^1} \Lambda\subset\R^{2n+3}$ from a closed, orientable Legendrian submanifold $\Lambda\subset\R^{2n+1}$. In Section~\ref{constrsmfrspinfd} we introduce a notion of front $S^m$-spinning. It generalizes the $S^1$-spinning construction and produces a closed, orientable Legendrian submanifold $\Sigma_{S^m} \Lambda\subset\R^{2(n+m)+1}$ from a closed, orientable Legendrian submanifold $\Lambda\subset\R^{2n+1}$. We have to mention that the special case of the ``higher dimensional spinning''
construction has already appeared in the work of Ekholm, Etnyre and Sabloff, see \cite{EkholmEtnyreSabloff}. In Section~\ref{sectbehcoblend} we investigate the behavior of the front $S^{m}$-spinning under the relation of an embedded Lagrangian cobordism.

Observe that a closed, orientable Legendrian submanifold $\Lambda \subset \R^{2n+1} $  gives rise to an exact Lagrangian cylinder $C(\Lambda) = \R \times \Lambda\subset (\R\times \R^{2n+1}, d(e^t\alpha))$.

\begin{definition}
A \emph{Lagrangian} (an \emph{exact Lagrangian}) \emph{cobordism} $L$ between two closed, orientable Legendrian submanifolds $\Lambda_{-}, \Lambda_{+}\subset \R^{2n+1}$ is an embedded Lagrangian (exact Lagrangian) submanifold in the symplectization of $\R^{2n+1}$ so that $L$ agrees with $C(\Lambda_{-})$
for $t \leq -T_{L}$, with $C(\Lambda_{+})$  for $t \geq T_{L}$, $L^{c}:=L|_{[-T_{L}-1,T_{L}+1]\times \R^{2n+1}}$ is compact for some $T_{L}\gg 0$ and we write $\Lambda_{-}\prec^{lag}_{L}\Lambda_{+}$ ($\Lambda_{-}\prec^{ex}_{L}\Lambda_{+}$). We will in general not distinguish between $L$ and $L^{c}$ and call both $L$. In the case when $\Lambda_{-}=\emptyset$, i.e. when $\Lambda_{+}$ has a Lagrangian (an exact Lagrangian) filling,   we write $\emptyset\prec^{lag}_{L}\Lambda_{+}$ ($\emptyset\prec^{ex}_{L}\Lambda_{+}$).
\end{definition}
From now on we assume that all Lagrangian cobordisms in the symplectization of $\R^{2n+1}$ are orientable and connected.

\begin{remark}
Observe that in the definition of an exact Lagrangian cobordism $L$ one usually requires  that
there is $f:L\to \R$ such that
$e^{t}\alpha|_{L} = df$ and $f_{-} := f|_{(-\infty,-T_{L}]\times \Lambda_{-}}$, $f_{+}:=f|_{[T_{L},\infty)\times \Lambda_{+}}$ are constant functions, see \cite{Chantraine2} and \cite{EkholmHondaKalman}.
In our settings, this condition is automatically satisfied because we consider only connected exact Lagrangian cobordisms with connected positive and negative ends.
\end{remark}

Our goal is to prove the following result:
\begin{proposition}\label{liftofcobordismsexandjustlagr}
Let $\Lambda_{-},\Lambda_{+}$ be two closed, orientable Legendrian submanifolds of $\R^{2n+1}$. If $\Lambda_{-}\prec_{L}^{lag}\Lambda_{+}$, then there exists a Lagrangian cobordism $\Sigma_{S^m} L$ such that
$\Sigma_{S^m} \Lambda_{-}\prec_{\Sigma_{S^m} L}^{lag}\Sigma_{S^m} \Lambda_{+}$. In addition, if $\Lambda_{-}\prec_{L}^{ex}\Lambda_{+}$, then there exists an exact Lagrangian cobordism $\Sigma_{S^m} L$ such that
$\Sigma_{S^m} \Lambda_{-}\prec_{\Sigma_{S^m} L}^{ex}\Sigma_{S^m} \Lambda_{+}$.
\end{proposition}

We then use the fact about the relation between the linearized Legendrian contact cohomology of a Legendrian submanifold of $\R^{2n+1}$ and the singular homology of its exact Lagrangian filling described by Ekholm in~\cite{Ekholm}, proven for $n=1$ by Ekholm, Honda and Kalman in~\cite{EkholmHondaKalman} and for all $n$ by Dimitroglou Rizell in~\cite{Rizell0},
and prove the following:

\begin{proposition}\label{legnotisotspinunspinhom}
Let $\Lambda_{-}$ and  $\Lambda_{+}$ be two closed, orientable Legendrian submanifolds of  $\R^{2n+1}$ such that $\emptyset\prec^{ex}_{L_{\Lambda_{-}}} \Lambda_{-}$ and
$\Lambda_{-}\prec^{ex}_{L} \Lambda_{+}$ with $\dim (H_{i}(L;\Z_{2}))>\dim(H_{i}(\Lambda_{-}; \Z_{2}))$ for some $i$. Then
\begin{itemize}
\item[(1)] $\Lambda_{-}$ is not Legendrian isotopic to $\Lambda_{+}$,
\item[(2)]$\Sigma_{S^{i_{k}}}\dots \Sigma_{S^{i_{1}}} \Lambda_{-}$ is not Legendrian isotopic to $\Sigma_{S^{i_{k}}}\dots \Sigma_{S^{i_{1}}} \Lambda_{+}$ for $i_{1},\dots,i_{k}
\geq i$, where $i$ is the smallest number such that $\dim (H_{i}(L;\Z_{2}))>\dim(H_{i}(\Lambda_{-}; \Z_{2}))$.
\end{itemize}
\end{proposition}

In Section~\ref{genfamaltappr} we say a few words about the way to get a variant of Proposition~\ref{legnotisotspinunspinhom} using the theory of generating families. For the basic definitions of this theory
we refer to \cite{SabloffTraynor}.

We apply Propositions~\ref{liftofcobordismsexandjustlagr} and \ref{legnotisotspinunspinhom} to a certain family of Legendrian knots and get the following theorem:
\begin{theorem}\label{thexamplessisposdifornotnlis}
There are infinitely many pairs of exact Lagrangian cobordant and not pairwise
Legendrian isotopic  Legendrian $S^1\times S^{i_{1}}\times \dots \times S^{i_{k}}$ in $\R^{2(\sum^{k}_{j=1}i_{j}+1)+1}$ which  have the same classical invariants if one of $i_{j}$'s is odd.
\end{theorem}

\section{Construction}\label{constrsmfrspinfd}

In this section we define a notion of front $S^m$-spinning. It is a natural generalization of the front $S^1$-spinning invented by Ekholm, Etnyre and Sullivan in~\cite{EkholmEtnyreSullivan}.

Let $\Lambda$ be a closed, orientable Legendrian submanifold of $\R^{2n+1}$ parameterized by $f_{\Lambda}:\Lambda\to \R^{2n+1}$ with
\begin{align*}
f_{\Lambda}(p)=(x_{1}(p),y_{1}(p),\dots,x_{n}(p),y_{n}(p),z(p))
\end{align*}
for $p\in \Lambda$. Without loss of generality assume that $x_{1}(p)>0$ for all $p$.



We define $\Sigma_{S^m} \Lambda$ to be the Legendrian submanifold of $\R^{2(m+n)+1}$ parameterized by
$f_{\Sigma_{S^m} \Lambda}:\Lambda\times S^{m}\to \R^{2(n+m)+1}$ with
\begin{align*}
f_{\Sigma_{S^m} \Lambda}(p,\theta,\overline{\phi})=
(\tilde{x}_{-m+1}(p,\theta,\overline{\phi}),\tilde{y}_{-m+1}(p,\theta,\overline{\phi})\dots,\tilde{x}_{1}(p,\theta,\overline{\phi}),\tilde{y}_{1}(p,\theta,\overline{\phi}),x_{2}(p),\dots,z(p)),
\end{align*}
where
\begin{align}\label{longxs}
\left \{
\begin{array}{l}
\tilde{x}_{-m+1}(p,\theta,\overline{\phi})=x_{1}(p)\sin\theta\sin\phi_{1}\dots\sin\phi_{m-1},\\
\tilde{x}_{-m+2}(p,\theta,\overline{\phi})=x_{1}(p)\cos\theta\sin\phi_{1}\dots\sin\phi_{m-1},\\
\tilde{x}_{-m+3}(p,\theta,\overline{\phi})=x_{1}(p)\cos\phi_{1}\dots\sin\phi_{m-1},\\
\dots\\
\tilde{x}_{1}(p,\theta,\overline{\phi})=x_{1}(p)\cos\phi_{m-1},
\end{array}
\right.
\end{align}
\begin{align}\label{longys}
\left \{
\begin{array}{l}
\tilde{y}_{-m+1}(p,\theta,\overline{\phi})=y_{1}(p)\sin\theta\sin\phi_{1}\dots\sin\phi_{m-1},\\
\tilde{y}_{-m+2}(p,\theta,\overline{\phi})=y_{1}(p)\cos\theta\sin\phi_{1}\dots\sin\phi_{m-1},\\
\tilde{y}_{-m+3}(p,\theta,\overline{\phi})=y_{1}(p)\cos\phi_{1}\dots\sin\phi_{m-1},\\
\dots\\
\tilde{y}_{1}(p,\theta,\overline{\phi})=y_{1}(p)\cos\phi_{m-1},
\end{array}
\right.
\end{align}
$\theta\in [0,2\pi)$ and $\overline{\phi}=(\phi_{1},\dots,\phi_{m-1})\in [0,\pi]^{m-1}$.

Since $\Lambda$ is a Legendrian submanifold of $\R^{2n+1}$ and hence $f_{\Lambda}^{\ast}(dz-\sum_{i=1}^{n}y_{i}dx_{i})=0$, we use Formulas~\ref{longxs} and \ref{longys} and  see that
\begin{align*}
f_{\Sigma_{S^m} \Lambda}^{\ast}(dz-\sum_{i=-m+1}^{n}y_{i}dx_{i})=0.
\end{align*}
Since $f_{\Lambda}(p)=(x_{1}(p),\dots,y_{n}(p),z(p))$,
where $p\in \Lambda$, is a parametrization of an embedded $n$-dimensional submanifold and $x_{1}(p)>0$ for all $p\in \Lambda$, we easily see that
\begin{align}\label{parammends}
f_{\Sigma_{S^m} \Lambda}(p,\theta,\overline{\phi})=
(\tilde{x}_{-m+1}(p,\theta,\overline{\phi}),\tilde{y}_{-m+1}(p,\theta,\overline{\phi}),\dots,\tilde{y}_{1}(p,\theta,\overline{\phi}),x_{2}(p),\dots,z(p))
\end{align}
with $\theta\in [0,2\pi)$, $\overline{\phi}=(\phi_{1},\dots,\phi_{m-1})\in [0,\pi]^{m-1}$ is a parametrization of an embedded $(n+m)$-dimensional submanifold.

\section{Proof of Proposition~\ref{liftofcobordismsexandjustlagr}}\label{sectbehcoblend}
Here we prove Proposition~\ref{liftofcobordismsexandjustlagr} by mimicking the proof of Proposition 1.5 from~\cite{Golovko}.

Given two closed, orientable Legendrian submanifolds $\Lambda_{-}, \Lambda_{+}\subset\R^{2n+1}$ such that
\begin{align}\label{poslegtotransl}
\Lambda_{\pm}\subset \{(x_{1},y_{1},\dots,x_{n},y_{n},z)\in \R^{2n+1}\ |\ x_{1}>0\}
\end{align}
and
$\Lambda_{-}\prec^{lag}_{L}\Lambda_{+}$. Let $f_{L}:L\to \R^{2n+2}$ be a parametrization of $L$
 with
\begin{align}\label{paramnoflstr}
f_{L}(p)=(t(p),x_{1}(p),y_{1}(p),\dots,x_{n}(p),y_{n}(p),z(p)),
\end{align}
where $p\in L$.
Assume that $x_{1}(p)>0$ for all $p$ (Formula~\ref{poslegtotransl} implies that $\{f_{L}(p)\ |\ x_{1}(p)\leq 0\}$ is compact and hence can be translated in such a way that $x_{1}(p)>0$ for all $p$).
We now construct a Lagrangian cobordism $\Sigma_{S^{m}} L$ from $\Sigma_{S^{m}} \Lambda_{-}$ to $\Sigma_{S^{m}} \Lambda_{+}$. Define $\Sigma_{S^{m}} L$ to be parametrized by
$f_{\Sigma_{S^{m}} L}: L\times S^m\to \R\times \R^{2(n+m)+1}$ with
\begin{align*}
f_{\Sigma_{S^{m}} L}(p,\theta,\overline{\phi})=(t(p),\tilde{x}_{-m+1}(p,\theta,\overline{\phi}),\tilde{y}_{-m+1}(p,\theta,\overline{\phi}),\dots,\tilde{x}_{1}(p,\theta,\overline{\phi}),\tilde{y}_{1}(p,\theta,\overline{\phi}),x_{2}(p),\dots,z(p)),
\end{align*}
where $\tilde{x}_{i}$'s and $\tilde{y}_{i}$'s are defined by Formulas~\ref{longxs} and \ref{longys} for $x_{i}$'s and $y_{i}$'s from Formula \ref{paramnoflstr}, $p\in L$, $\theta\in [0,2\pi)$ and $\overline{\phi}=(\phi_{1},\dots,\phi_{m-1})\in [0,\pi]^{m-1}$.

Here we show that $\Sigma_{S^{m}} L$ is a Lagrangian cobordism from $\Sigma_{S^m} \Lambda_{-}$ to $\Sigma_{S^m} \Lambda_{+}$. We first note that $f_{\Sigma_{S^{m}} L}(L\times S^m)\cap \{t=t_{0}\}$ for $t_{0}\geq T_{L}$ (or $t_{0}\leq -T_{L}$) can be parametrized by $f_{\Sigma_{S^{m}}\Lambda_{+}}(p,\theta,\overline{\phi})$ (or $f_{\Sigma_{S^{m}}\Lambda_{-}}(p,\theta,\overline{\phi})$) for $f_{\Sigma_{S^{m}}\Lambda_{\pm}}(p,\theta,\overline{\phi})$ from Formula~\ref{parammends}, where $p\in \Lambda_{+}\subset \partial L^{c}$ (or $p\in \Lambda_{-}\subset \partial L^{c}$), $\theta\in [0,2\pi)$ and $\overline{\phi}=(\phi_{1},\dots,\phi_{m-1})\in [0,\pi]^{m-1}$. In addition, from the fact that $L^{c}$ is compact it follows that $\Sigma_{S^{m}} L^{c}$ is also compact. It remains to prove that the cobordism $\Sigma_{S^{m}}L$ is an embedded Lagrangian cobordism.

Since since $L$ is a Lagrangian cobordism and hence
\begin{align*}
f_{L}^{\ast}(d(e^t(dz-\sum^{n}_{i=1}y_{i}dx_{i})))=0
\end{align*}
we use simple trigonometric identities and  get that
\begin{align}\label{pullbsympln}
f_{\Sigma_{S^{m}} L}^{\ast}(d(e^t(dz-\sum^{n}_{i=-m+1}y_{i}dx_{i})))=0.
\end{align}

Since $f_{L}(p)=(t(p),x_{1}(p),y_{1}(p),\dots,x_{n}(p),y_{n}(p),z(p))$,
where $p\in L$, is a parametrization of an embedded $(n+1)$-dimensional cobordism and $x_{1}(p)>0$ for $p\in L$, one easily sees that
\begin{align*}
f_{\Sigma_{S^{m}} L}(p,\theta,\overline{\phi})=(t(p),\tilde{x}_{-m+1}(p,\theta,\overline{\phi}),\tilde{y}_{-m+1}(p,\theta,\overline{\phi}),\dots,\tilde{x}_{1}(p,\theta,\overline{\phi}),\tilde{y}_{1}(p,\theta,\overline{\phi}),x_{2}(p),\dots,z(p)),
\end{align*}
where $p\in L$, $\theta\in [0,2\pi)$ and $\overline{\phi}=(\phi_{1},\dots,\phi_{m-1})\in [0,\pi]^{m-1}$, is a parametrization of an embedded $(n+m+1)$-dimensional cobordism. Thus we use Formula~\ref{pullbsympln} and get that $\Sigma_{S^{m}} L$ is really an embedded Lagrangian cobordism from $\Sigma_{S^{m}} \Lambda_{-}$ to $\Sigma_{S^{m}} \Lambda_{+}$.

We now assume that $\Lambda_{-}\prec^{ex}_{L}\Lambda_{+}$. Then there exists a function $h_{L}\in C^{\infty}(f_{L}(L),\R)$ such that
\begin{align*}
dh_{L}=e^t(dz-\sum^{n}_{i=1}y_{i}dx_{i}).
\end{align*}
Since $f_{\Sigma_{S^{m}} L}$ is an embedding, we can define $h_{\Sigma_{S^m} L}\in C^{\infty}(f_{\Sigma_{S^m} L}(\Sigma_{S^m} L),\R)$ by setting
\begin{align*}
(f_{\Sigma_{S^m} L}^{\ast}h_{\Sigma_{S^m} L})(p,\theta,\overline{\phi}):=(f_{L}^{\ast}h_{L})(p).
\end{align*}
Observe that
\begin{align*}
f_{\Sigma_{S^m}L}^{\ast}(e^t(dz-\sum^{n}_{i=-m+1}y_{i}dx_{i}))=e^{t(p)}(dz(p)-\sum^{n}_{i=1}y_{i}(p)dx_{i}(p)).
\end{align*}
Therefore we get that
\begin{align}\label{exactnesslongformpushed}
d(f_{\Sigma_{S^m} L}^{\ast}h_{\Sigma^{m} L})=e^{t(p)}(dz(p)-\sum^{n}_{i=1}y_{i}(p)dx_{i}(p))=f_{\Sigma_{S^m} L}^{\ast}(e^t(dz-\sum^{n}_{i=-m+1}y_{i}dx_{i})).
\end{align}
$f_{\Sigma_{S^m} L}$ is an embedding and hence Formula~\ref{exactnesslongformpushed} implies that
\begin{align*}
d(h_{\Sigma_{S^m} L})=e^t(dz-\sum^{n}_{i=-m+1}y_{i}dx_{i}).
\end{align*}
Hence $\Sigma_{S^{m}} L$ is an exact Lagrangian cobordism.
This finishes the proof of Proposition~\ref{liftofcobordismsexandjustlagr}.

Observe that the proof of Proposition~\ref{liftofcobordismsexandjustlagr} can be easily transformed to become a proof of the following remark:
\begin{remark}\label{thmfiltb}
Given a  closed, orientable Legendrian submanifold $\Lambda\subset\R^{2n+1}$. If $\emptyset\prec_{L_{\Lambda}}^{lag}\Lambda$ ($\emptyset\prec_{L_{\Lambda}}^{ex}\Lambda$), then there exists a Lagrangian (an exact Lagrangian) filling $L_{\Sigma_{S^m}\Lambda}$ such that
$\emptyset \prec_{L_{\Sigma_{S^m}\Lambda}}^{lag}\Sigma_{S^m} \Lambda$ ($\emptyset\prec_{L_{\Sigma_{S^m}\Lambda}}^{ex}\Sigma_{S^m} \Lambda$).
\end{remark}

\section{Proof of Proposition~\ref{legnotisotspinunspinhom}}
Let $\Lambda_{-}$ and $\Lambda_{+}$ be two closed, orientable Legendrian submanifolds of $\R^{2n+1}$ with $\emptyset\prec^{ex}_{L_{\Lambda_{-}}}\Lambda_{-}$, $\Lambda_{-}\prec_{L}^{ex} \Lambda_{+}$ and
\begin{align}\label{diffdimshicobandendleg}
\dim(H_{i}(L; \Z_{2}))>\dim(H_{i}(\Lambda_{-}; \Z_{2}))
\end{align}
for some $i$.

We first construct an exact Lagrangian filling of $\Lambda_{+}$.
Since $\Lambda_{-}$ is connected, and $L$, $L_{\Lambda_{-}}$ are exact Lagrangian cobordisms in the symplectization of $\R^{2n+1}$, we
glue the positive end of $L_{\Lambda_{-}}$ to the negative
end of $L$ and get an exact Lagrangian filling $L_{\Lambda_{+}}$ of $\Lambda_{+}$.

Consider the Mayer-Vietoris long exact sequence for $L_{\Lambda_{-}}$, $L$ and $L_{\Lambda_{+}}$
\begin{align}\label{longexseqmavjethefirstone}
\dots \to H_{i}(\Lambda_{-}; \Z_{2})\to H_{i}(L; \Z_{2})\oplus H_{i}(L_{\Lambda_{-}}; \Z_{2})\xrightarrow{\Phi} H_{i}(L_{\Lambda_{+}}; \Z_{2}) \to \dots.
\end{align}
From Formulas~\ref{diffdimshicobandendleg} and \ref{longexseqmavjethefirstone} it follows that
\begin{align*}
\dim(H_{i}(L_{\Lambda_{+}};\Z_{2})) &\geq \dim(\Phi(H_{i}(L; \Z_{2})\oplus H_{i}(L_{\Lambda_{-}}; \Z_{2}))) \\
&\geq  \dim(H_{i}(L_{\Lambda_{-}};\Z_{2})) + \dim(H_{i}(L;\Z_{2})) - \dim(H_{i}(\Lambda_{-}; \Z_{2}))\nonumber \\
&> \dim(H_{i}(L_{\Lambda_{-}};\Z_{2})).
\end{align*}
Hence we get that
\begin{align}\label{finalineqfrlefttorthilam}
\dim(H_{i}(L_{\Lambda_{+}};\Z_{2}))>
\dim(H_{i}(L_{\Lambda_{-}};\Z_{2})).
\end{align}

We now define the following notations. If $\Lambda$ is a closed, orientable Legendrian submanifold of $\R^{2n+1}$, then
we denote by $\L(\Lambda)$ the set of all embedded exact Lagrangian fillings of $\Lambda$ and
\begin{align*}
\mathcal H^{geom}_{\Lambda}:=\{ (\dim(H_{i}(L_{\Lambda}; \Z_{2})))_{i}: L_{\Lambda}\in \L(\Lambda)\}.
\end{align*}

Here we remind the reader of the following isomorphism described by Ekholm in~\cite{Ekholm}, which comes from certain observations of Seidel in wrapped Floer homology~\cite{AbouzaidSeidel}, \cite{FukayaSeidelSmith}. Note that the existence of this isomorphism has been proven for $n=1$ by Ekholm, Honda and Kalman in~\cite{EkholmHondaKalman} and for all $n$ by Dimitroglou Rizell in~\cite{Rizell0}.

\begin{theorem}[\cite{Ekholm}, \cite{EkholmHondaKalman}, \cite{Rizell0}]\label{bigproblhopenotcomplproven}
Let $\Lambda$ be a closed, orientable Legendrian submanifold of $\R^{2n+1}$ and $\emptyset\prec^{ex}_{L_{\Lambda}}\Lambda$. Then
\begin{align*}
H_{n-i}(L_{\Lambda}; \Z_{2})\simeq LCH_{\ve}^{i}(\Lambda).
\end{align*}
Here $\ve$ is the augmentation induced by $L_{\Lambda}$.
\end{theorem}

\begin{remark}\label{stbutndchofgrsmaslov}
Note that in order for gradings in Theorem \ref{bigproblhopenotcomplproven} to be well defined, one should assume that the Maslov class of $L_{\Lambda}$ is trivial and the Maslov number of $\Lambda$ is zero. If one does not make these assumptions, then from discussion in \cite{Rizell0} it follows that the following formula holds
\begin{align*}
\dim(\ker(\partial_{\ve})) - \dim(Im (\partial_{\ve})) = \sum_{i}
\dim (H_{i}(L_{\Lambda}; \Z_{2})).
\end{align*}
Here $\partial_{\ve}$ is a differential of $LC_{\ve}(\Lambda)$.
\end{remark}

From
the fact that Legendrian isotopy implies that there exists an exact Lagrangian cylinder, see Proposition 1.4 in \cite{Golovko}, it follows that $\mathcal H^{geom}_{\Lambda}$ is a Legendrian invariant.

Note that Formula~\ref{finalineqfrlefttorthilam} holds
for every exact Lagrangian filling $L_{\Lambda_{+}}$ of $\Lambda_{+}$ obtained by gluing the positive end of an exact Lagrangian filling $L_{\Lambda_{-}}$ to the negative end
of $L$.

Since $\Lambda_{-}$ is chord generic, every linearized Legendrian contact cohomology
complex of $\Lambda_{-}$ has the same (finite) number of generators.  Therefore using that $\emptyset\prec^{ex}_{L_{\Lambda_{-}}}\Lambda_{-}$, Theorem~\ref{bigproblhopenotcomplproven} and Remark~\ref{stbutndchofgrsmaslov}
we get that there is $L^{max}_{\Lambda_{-}}\in \L(\Lambda_{-})$ such that
\begin{align}\label{maxiforhomolgybothends}
\dim(H_{i}(L^{max}_{\Lambda_{-}}; \Z_{2}))\geq \dim(H_{i}(L_{\Lambda_{-}}; \Z_{2}))
\end{align}
for all $L_{\Lambda_{-}}\in \L(\Lambda_{-})$.
Then we construct $L^{sep}_{\Lambda_{+}}$ which is an exact Lagrangian cobordism obtained by gluing the positive end of $L^{max}_{\Lambda_{-}}$ to the negative
end of $L$. Formulas~\ref{finalineqfrlefttorthilam} and \ref{maxiforhomolgybothends} imply that
\begin{align*}
\dim(H_{i}(L^{sep}_{\Lambda_{+}};\Z_{2}))>\dim(H_{i}(L^{max}_{\Lambda_{-}}; \Z_{2}))
\end{align*}
and hence
$\mathcal H^{geom}_{\Lambda_{-}}\neq \mathcal H^{geom}_{\Lambda_{+}}$. Thus $\Lambda_{-}$ is not Legendrian isotopic to $\Lambda_{+}$. This finishes the proof of the first part of Proposition~\ref{legnotisotspinunspinhom}.

Consider $\Sigma_{S^{i_{k}}}\dots\Sigma_{S^{i_{1}}} L$ and $\Sigma_{S^{i_{k}}}\dots\Sigma_{S^{i_{1}}} \Lambda_{-}$, where $i_{1},\dots,i_{k}
\geq i$ and $i$ is the smallest number such that $\dim (H_{i}(L;\Z_{2}))>\dim(H_{i}(\Lambda_{-}; \Z_{2}))$.
Observe that $\Sigma_{S^{i_{k}}}\dots\Sigma_{S^{i_{1}}} L$ is diffeomorphic to $L\times S^{i_{1}}\times \dots\times S^{i_{k}}$ and $\Sigma_{S^{i_{k}}}\dots\Sigma_{S^{i_{1}}} \Lambda_{-}$ is diffeomorphic to $\Lambda_{-}\times S^{i_{1}}\times \dots\times S^{i_{k}}$.  Hence
\begin{align}\label{homiintermcobismwithdiffis}
&H_{i}(\Sigma_{S^{i_{k}}}\dots\Sigma_{S^{i_{1}}} L; \Z_{2})\simeq H_{i}(\Sigma_{S^{i_{k-1}}}\dots\Sigma_{S^{i_{1}}} L; \Z_{2})\oplus H_{i-i_{k},\ 0}(\Sigma_{S^{i_{k-1}}}\dots\Sigma_{S^{i_{1}}} L; \Z_{2})\\
&\simeq H_{i}(L;\Z_{2})\oplus H_{i-i_{1},\ 0}(L;\Z_{2})\oplus\dots\oplus H_{i-i_{k},\ 0}(\Sigma_{S^{i_{k-1}}}\dots\Sigma_{S^{i_{1}}} L; \Z_{2})\simeq H_{i}(L;\Z_{2})\oplus \Z_{2}^{l},\nonumber
\end{align}
where
\begin{align*}
H_{i-i_{j},\ 0}(\ \cdot \ ; \Z_{2})\simeq
\left \{
\begin{array}{ll}
H_{i-i_{j}}(\ \cdot \ ; \Z_{2}), & \mbox{if}\ i-i_{j}\geq 0;\ \\
0, & \mbox{if}\ i-i_{j}<0
\end{array}
\right.
\end{align*}
and $l$ is the number of $j$'s such that $i_{j}=i$.
Similarly, we see that
\begin{align}\label{homsigmaendsagdiffpl}
H_{i}(\Sigma_{S^{i_{k}}}\dots\Sigma_{S^{i_{1}}} \Lambda_{-}; \Z_{2})&\simeq H_{i}(\Lambda_{-};\Z_{2})\oplus \Z_{2}^{l},
\end{align}
where $l$ is the number of $j$'s such that $i_{j}=i$. Observe that here we use that $\Lambda_{-}$ and $L$ are connected.
Formulas~\ref{diffdimshicobandendleg}, \ref{homiintermcobismwithdiffis} and \ref{homsigmaendsagdiffpl} imply that
\begin{align}\label{serineqaftliftwithiscond}
\dim(H_{i}(\Sigma_{S^{i_{k}}}\dots\Sigma_{S^{i_{1}}} L; \Z_{2})) &= \dim(H_{i}(L;\Z_{2})) + l > \dim(H_{i}(\Lambda_{-};\Z_{2})) + l\\
&= \dim(H_{i}(\Sigma_{S^{i_{k}}}\dots\Sigma_{S^{i_{1}}} \Lambda_{-}; \Z_{2})) \nonumber.
\end{align}
Observe that from the proof of Proposition~\ref{liftofcobordismsexandjustlagr} it follows that
\begin{align*}
\Sigma_{S^{i_{k}}}\dots\Sigma_{S^{i_{1}}} \Lambda_{-} \prec^{ex}_{\Sigma_{S^{i_{k}}}\dots\Sigma_{S^{i_{1}}} L} \Sigma_{S^{i_{k}}}\dots\Sigma_{S^{i_{1}}} \Lambda_{+}.
\end{align*}

We now take a chord generic representative $(\Sigma_{S^{i_{k}}}\dots\Sigma_{S^{i_{1}}} \Lambda_{\pm})_{gen}$ in the Legendrian isotopy class of $\Sigma_{S^{i_{k}}}\dots\Sigma_{S^{i_{1}}} \Lambda_{\pm}$. Note that from the proof of Proposition 1.4 in \cite{Golovko} it follows that there are exact Lagrangian cobordisms $L^{(\Sigma_{S^{i_{k}}}\dots\Sigma_{S^{i_{1}}} \Lambda_{\pm})_{gen}}_{\Sigma_{S^{i_{k}}}\dots\Sigma_{S^{i_{1}}} \Lambda_{\pm}}$ from $\Sigma_{S^{i_{k}}}\dots\Sigma_{S^{i_{1}}} \Lambda_{\pm}$ to $(\Sigma_{S^{i_{k}}}\dots\Sigma_{S^{i_{1}}} \Lambda_{\pm})_{gen}$ and $L_{(\Sigma_{S^{i_{k}}}\dots\Sigma_{S^{i_{1}}} \Lambda_{\pm})_{gen}}^{\Sigma_{S^{i_{k}}}\dots\Sigma_{S^{i_{1}}} \Lambda_{\pm}}$ from $(\Sigma_{S^{i_{k}}}\dots\Sigma_{S^{i_{1}}} \Lambda_{\pm})_{gen}$ to $\Sigma_{S^{i_{k}}}\dots\Sigma_{S^{i_{1}}} \Lambda_{\pm}$ that are diffeomorphic to $\R\times \Sigma_{S^{i_{k}}}\dots\Sigma_{S^{i_{1}}} \Lambda_{\pm}$.
Hence we use the fact that $\emptyset\prec_{\Sigma_{S^{i_{k}}}\dots\Sigma_{S^{i_{1}}} L_{\Lambda_{-}}}^{ex} \Sigma_{S^{i_{k}}}\dots\Sigma_{S^{i_{1}}} \Lambda_{-}$, glue $L^{(\Sigma_{S^{i_{k}}}\dots\Sigma_{S^{i_{1}}} \Lambda_{-})_{gen}}_{\Sigma_{S^{i_{k}}}\dots\Sigma_{S^{i_{1}}} \Lambda_{-}}$ to $\Sigma_{S^{i_{k}}}\dots\Sigma_{S^{i_{1}}} L_{\Lambda_{-}}$ along $\Sigma_{S^{i_{k}}}\dots\Sigma_{S^{i_{1}}}\Lambda_{-}$ and get $ L_{(\Sigma_{S^{i_{k}}}\dots\Sigma_{S^{i_{1}}}\Lambda_{-})_{gen}}$ which is an exact Lagrangian filling  of $(\Sigma_{S^{i_{k}}}\dots\Sigma_{S^{i_{1}}} \Lambda_{-})_{gen}$. In addition, if we glue $\Sigma_{S^{i_{k}}}\dots\Sigma_{S^{i_{1}}} L$ to $L_{(\Sigma_{S^{i_{k}}}\dots\Sigma_{S^{i_{1}}} \Lambda_{-})_{gen}}^{\Sigma_{S^{i_{k}}}\dots\Sigma_{S^{i_{1}}} \Lambda_{-}}$
along $\Sigma_{S^{i_{k}}}\dots\Sigma_{S^{i_{1}}} \Lambda_{-}$ and to
$L^{(\Sigma_{S^{i_{k}}}\dots\Sigma_{S^{i_{1}}} \Lambda_{+})_{gen}}_{\Sigma_{S^{i_{k}}}\dots\Sigma_{S^{i_{1}}} \Lambda_{+}}$
along $\Sigma_{S^{i_{k}}}\dots\Sigma_{S^{i_{1}}} \Lambda_{+}$, then
we get an exact Lagrangian cobordism $L_{(\Sigma_{S^{i_{k}}}\dots\Sigma_{S^{i_{1}}} \Lambda_{-})_{gen}}^{(\Sigma_{S^{i_{k}}}\dots\Sigma_{S^{i_{1}}} \Lambda_{+})_{gen}}$ from $(\Sigma_{S^{i_{k}}}\dots\Sigma_{S^{i_{1}}} \Lambda_{-})_{gen}$ to $(\Sigma_{S^{i_{k}}}\dots\Sigma_{S^{i_{1}}} \Lambda_{+})_{gen}$ which is diffeomorphic to $\Sigma_{S^{i_{k}}}\dots\Sigma_{S^{i_{1}}} L$.
From Formula~\ref{serineqaftliftwithiscond} it follows that
\begin{align}
\dim(H_{i}(L_{(\Sigma_{S^{i_{k}}}\dots\Sigma_{S^{i_{1}}} \Lambda_{-})_{gen}}^{(\Sigma_{S^{i_{k}}}\dots\Sigma_{S^{i_{1}}} \Lambda_{+})_{gen}}; \Z_{2})) >  \dim(H_{i}((\Sigma_{S^{i_{k}}}\dots\Sigma_{S^{i_{1}}} \Lambda_{-})_{gen}; \Z_{2})) \nonumber.
\end{align}
Hence the first part of Proposition~\ref{legnotisotspinunspinhom} implies that
$(\Sigma_{S^{i_{k}}}\dots\Sigma_{S^{i_{1}}} \Lambda_{-})_{gen}$ is not Legendrian isotopic to $(\Sigma_{S^{i_{k}}}\dots\Sigma_{S^{i_{1}}} \Lambda_{+})_{gen}$. Therefore
$\Sigma_{S^{i_{k}}}\dots\Sigma_{S^{i_{1}}} \Lambda_{-}$ is not Legendrian isotopic to $\Sigma_{S^{i_{k}}}\dots\Sigma_{S^{i_{1}}} \Lambda_{+}$.
This finishes the proof of Proposition~\ref{legnotisotspinunspinhom}.

\section{Proof of Theorem~\ref{thexamplessisposdifornotnlis}}
We first prove the following simple lemma:
\begin{lemma}\label{clasprosofrotnnonstdbutstag}
Let $\Lambda$ be a  closed, orientable Legendrian submanifold of $\R^{2n+1}$ such that $\emptyset\prec^{ex}_{L_{\Lambda}}\Lambda$.
Then $\Sigma_{S^m} \Lambda$ has
\begin{itemize}
\item[(1)] the topological type of $\Lambda\times S^m$,
\item[(2)] the rotation class of $\Sigma_{S^m} \Lambda$ is determined by the rotation class of $\Lambda$, and
\item[(3)] the Thurston-Bennequin number
\begin{align*}
tb(\Sigma_{S^m}\Lambda)=
\left \{
\begin{array}{ll}
2(-1)^{\frac{m}{2}}tb(\Lambda), & \mbox{if}\ m\ \mbox{is even};\ \\
0, & \mbox{if}\ m\ \mbox{is odd}.
\end{array}
\right.
\end{align*}
\end{itemize}
\end{lemma}
\begin{proof}
We first observe that $(1)$ and $(2)$ are straightforward and follow from the construction of $\Sigma_{S^m} \Lambda$.
Then we prove $(3)$. Note that $L_{\Sigma_{S^m}\Lambda}:= \Sigma_{S^m}L_{\Lambda}$ is diffeomorphic to $L_{\Lambda}\times S^m$. Hence  we get
\begin{align}\label{eulcharsnrotexlagrcob}
\chi(L_{\Sigma_{S^m}\Lambda})=\chi(L_{\Lambda}\times S^m)=\chi(L_{\Lambda})\chi(S^m)=\left \{
\begin{array}{ll}
2\chi(L_{\Lambda}), & \mbox{if}\ m\ \mbox{is even};\\
0, & \mbox{if}\ m\ \mbox{is odd} .
\end{array}
\right.
\end{align}
We now recall that
\begin{align}\label{oldtbfingothpap}
&tb(\Lambda)=
\left \{
\begin{array}{ll}
(-1)^{\frac{n}{2}+1}\chi(L_{\Lambda}), & \mbox{if}\ n\ \mbox{is even};\\
(-1)^{\frac{(n-2)(n-1)}{2}+1}\chi(L_{\Lambda}), & \mbox{if}\ n\ \mbox{is odd},
\end{array}
\right.
\end{align}
see Remark 3.5 in \cite{Golovko}.
Formulas  \ref{eulcharsnrotexlagrcob} and \ref{oldtbfingothpap} imply that
\begin{align}
tb(\Sigma_{S^m}\Lambda)&=
\left \{
\begin{array}{ll}
(-1)^{\frac{m+n}{2}+1}\chi(L_{\Sigma_{S^m}\Lambda}), & \mbox{if}\ n, m\ \mbox{are even};\\
(-1)^{\frac{(n+m-2)(n+m-1)}{2}+1}\chi(L_{\Sigma_{S^m}\Lambda}), & \mbox{if}\ n\ \mbox{is odd}, m\ \mbox{is even};\\
0, & \mbox{if}\ m\ \mbox{is odd};
\end{array}
\right.\\ \nonumber
&= \left \{
\begin{array}{ll}
2(-1)^{\frac{m}{2}}tb(\Lambda), & \mbox{if}\ m\ \mbox{is even};\ \\
0, & \mbox{if}\ m\ \mbox{is odd}.
\end{array}
\right.
\end{align}
\end{proof}

\begin{figure}[t]
\centering
\includegraphics[width=250pt]{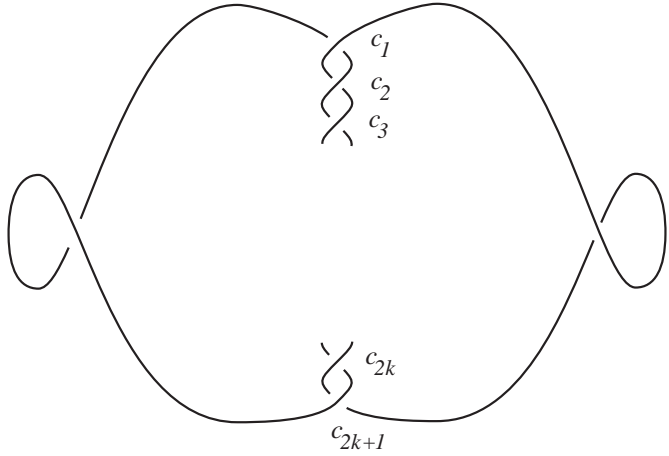}
\caption{The knot $T_{2k+1}$.}
\label{nonlegisknotsexamples}
\end{figure}

Let $T_{2k+1}$ be the Legendrian torus knot whose Lagrangian projection is in Figure~\ref{nonlegisknotsexamples} with rotation number $r(T_{2k+1})=0$ for $k\geq 1$. Note that $T_{2k+1}$ admits an exact Lagrangian filling for every $k\geq 1$, see Section 8.1 in \cite{EkholmHondaKalman}.
In addition, recall that $T_{2j+1}\prec_{L^{2k+1}_{2j+1}}^{ex}T_{2k+1}$ for $k>j$, see the proof of Proposition 1.6 in \cite{Golovko} or Section 8.1 in \cite{EkholmHondaKalman}. It is easy to see that $tb(T_{2k+1})=2k-1$.  We use Theorem 1.2 from \cite{Chantraine} and get that
\begin{align}
 2(k-j) = tb(T_{2k+1}) - tb(T_{2j+1}) = - \chi (L^{2k+1}_{2j+1}) = 2g(L^{2k+1}_{2j+1}).
\end{align}
Therefore
$L^{2k+1}_{2j+1}$ is a twice punctured genus $g(L^{2k+1}_{2j+1})=k-j$ oriented surface. Hence we see that
\begin{align*}
\dim (H_{1}(L^{2k+1}_{2j+1}; \Z_{2})) = 2(k-j)+1>1=\dim (H_{1}(T_{2j+1}; \Z_{2}))
\end{align*}
for $k>j$.

We then apply Propositions~\ref{liftofcobordismsexandjustlagr} and \ref{legnotisotspinunspinhom} and get that there are infinitely many pairs of exact Lagrangian cobordant  and not pairwise Legendrian isotopic  Legendrian $S^1\times S^{i_{1}}\times \dots \times S^{i_{k}}$ in $\R^{2(\sum^{k}_{j=1}i_{j}+1)+1}$.
In addition, from Lemma~\ref{clasprosofrotnnonstdbutstag} it follows that the classical invariants of $S^1\times S^{i_{1}}\times \dots \times S^{i_{k}}$'s agree if one of $i_{j}$'s is odd.
This finishes the proof of Theorem~\ref{thexamplessisposdifornotnlis}.

\section{Alternative approach}\label{genfamaltappr}
One can use generating family cohomology (over $\Z_{2}$) instead of linearized Legendrian contact cohomology to prove a variant of Proposition~\ref{legnotisotspinunspinhom} (for the basic definitions of the theory of generating families
we refer to \cite{SabloffTraynor}):
\begin{proposition}\label{genfampropagliftandgen}
Let $\Lambda_{-}$ and  $\Lambda_{+}$ be two closed, orientable Legendrian submanifolds of  $J^1(M)$ such that $\emptyset \prec^{lag}_{L_{\Lambda_{-}}} \Lambda_{-}$,
$\Lambda_{-}\prec^{lag}_{L} \Lambda_{+}$ and $\dim (H_{i}(L;\Z_{2}))>\dim(H_{i}(\Lambda_{-}; \Z_{2}))$ for some $i$. In addition, assume that $L_{\Lambda_{-}}$ admits a tame, compatible triple of generating families $(F_{L_{\Lambda_{-}}}, f_{\emptyset}^{-}, f_{\Lambda_{-}}^{+})$ and $L$ is gf-compatible to $\Lambda_{-}$. Then
$\Lambda_{-}$ is not Legendrian isotopic to $\Lambda_{+}$.
\end{proposition}
Here $M$ is a compact manifold (or $\R^n$) and $J^1(M)$ is a $1$-jet space of $M$.
In addition, the property that $L$ is gf-compatible to $\Lambda_{-}$ means that for every tame generating family $f_{\Lambda_{-}}$ of $\Lambda_{-}$ there exists a tame, compatible triple of generating families  $(F_{L}, f^{-}_{\Lambda_{-}}, f^{+}_{\Lambda_{+}})$ for $L$, where $f^{-}_{\Lambda_{-}}$ and $f_{\Lambda_{-}}$ are in the same equivalence class (classes are defined up to stabilizations and fiber-preserving diffeomorphisms; for more details we refer to~\cite{SabloffTraynor}).

One proves Proposition~\ref{genfampropagliftandgen} by simply mimicking the proof of Proposition~\ref{legnotisotspinunspinhom} and using the following observations:
\begin {itemize}
\item[(i)] If we glue a Lagrangian filling $L_{\Lambda_{-}}$ of $\Lambda_{-}$ which admits a tame, compatible triple of generating families $(F_{L_{\Lambda_{-}}}, f^{-}_{\emptyset}, f^{+}_{\Lambda_{-}})$
 to $L$ which is gf-compatible to $\Lambda_{-}$ along $\Lambda_{-}$, then we get a Lagrangian filling $L_{\Lambda_{+}}$ of $\Lambda_{+}$ which admits a tame, compatible triple of generating families $(F_{L_{\Lambda_{+}}}, f^{-}_{\emptyset}, f^{+}_{\Lambda_{+}})$. The way to glue two cobordisms which admit tame, compactible triples of generating families is written in~\cite{SabloffTraynor}.
\item[(ii)]  Sabloff and Traynor in~\cite{SabloffTraynor} proved an  analoque of Theorem~\ref{bigproblhopenotcomplproven}, i.e., they prove that $GH^{k}(f_{\Lambda}^{+})\simeq H^{k+1}(L_{\Lambda}, \Lambda; \Z_{2})$ for a Lagrangian filling $L_{\Lambda}$ of a closed, orientable Legendrian submanifold $\Lambda$ with a  tame, compatible triple of generating families $(F_{L_{\Lambda}}, f^{-}_{\emptyset},f^{+}_{\Lambda})$.
\item[(iii)] $\dim(GH^{i}(f^{+}_{\Lambda_{-}}))\leq 2l_{-}$ for all $i$, where $l_{-}$ is the number of Reeb chords of $\Lambda_{-}$. This follows from the description of the critical points of the difference function, see \cite{FuchsRutherford}, \cite{SabloffTraynor0}. Hence for a fixed $i$ there exists an embedded Lagrangian filling $L^{max}_{\Lambda_{-}}$ of $\Lambda_{-}$  which admits
a tame, compatible triple of generating families $(F_{L^{max}_{\Lambda_{-}}}, f^{-}_{\emptyset, max}, f^{+}_{\Lambda_{-}, max})$
such that
\begin{align*}
\dim(H_{i}(L^{max}_{\Lambda_{-}};\Z_{2}))& = \dim(H^{n-i+1}(L^{max}_{\Lambda_{-}}, \Lambda_{-};\Z_{2})) = \dim(GH^{n-i}(f_{\Lambda_{-}, max}^{+}))\\
&\geq \dim(GH^{n-i}(f_{\Lambda_{-}}^{+})) = \dim(H_{i}(L_{\Lambda_{-}};\Z_{2}))
\end{align*}
for every embedded Lagrangian filling $L_{\Lambda_{-}}$ of $\Lambda_{-}$ which admits a tame, compatible triple of generating families $(F_{L_{\Lambda_{-}}}, f^{-}_{\emptyset}, f^{+}_{\Lambda_{-}})$.
\end{itemize}
\begin{remark}
Observe that Bourgeois, Sabloff and Traynor in \cite{BourgeoisSabloffTraynor} prove that the operation of ``standard'' Lagrangian handle attachment can be realized as a Lagrangian cobordism which is gf-compatible to its negative end. Hence every Lagrangian cobordism obtained by gluing Lagrangian cobordisms which correspond to ``standard'' Lagrangian handle attachments is also gf-compatible to its negative end.
\end{remark}

\section*{Acknowledgements}
The author is deeply grateful to Fr\'{e}d\'{e}ric
Bourgeois, Baptiste Chantraine, Octav Cornea, Tobias Ekholm,
Georgios Dimitroglou Rizell and Joshua Sabloff for helpful
conversations and interest in his work.

\end{document}